\documentclass[10pt,english]{article}
\usepackage{mathtools}
\usepackage{mathrsfs}
\usepackage{graphicx}
\usepackage{graphics}
\usepackage{setspace}
	\usepackage{color}
\usepackage{amsmath,amssymb,amsthm}
\usepackage{epsfig}
\usepackage{babel}
\usepackage{tikz}
\usepackage{tikz-cd}
\usetikzlibrary{arrows,shapes,automata,backgrounds,petri,positioning,scopes,matrix, calc, plothandlers,plotmarks,patterns}
\usepackage{enumerate}
\usepackage{float}
\usepackage[utf8]{inputenc}
\usepackage[T1]{fontenc}
\usepackage[document]{ragged2e}
 \usepackage{verbatim}
\usepackage{tikz-3dplot}
\usepackage{pgfplots}
\usepackage{tkz-graph}
\GraphInit[vstyle = Normal]
\tikzset{
  LabelStyle/.style = {minimum width = 2em, 
                        text = red, font = \bfseries },
  VertexStyle/.append style = { inner sep=2pt,
                                font = \Large\bfseries, fill},
  EdgeStyle/.append style = {->, bend left} }

\newtheorem{thm}{Theorem}[section]
\numberwithin{equation}{section} 
\numberwithin{figure}{thm} 
\theoremstyle{plain}
\newtheorem*{thm*}{Theorem}
\theoremstyle{definition}
\theoremstyle{plain}
\newtheorem{thm_A}{Theorem}
\newtheorem*{defn*}{Definition}

\theoremstyle{plain}

\theoremstyle{plain} 

\theoremstyle{plain}
\theoremstyle{definition}

\theoremstyle{remark}
\newtheorem{rem}[thm]{Remark}
\theoremstyle{plain}

\theoremstyle{plain}

\theoremstyle{plain}
\newtheorem{lem}[thm]{Lemma}
\newtheorem*{lem*}{Lemma} 
\theoremstyle{definition}

\newtheorem*{acknowledgment*}{Addentum}
\newtheorem{ques}[thm]{Question}
\theoremstyle{plain}
\newtheorem*{ex*}{Example}
\theoremstyle{plain}

\AtBeginDocument{
  
}
\newtheorem{thmx}{Theorem}

\begin{document}
\pgfdeclarelayer{background}
\pgfsetlayers{background,main}
\title{The space of coset partitions of $F_n$ and  Herzog-Sch\"onheim conjecture}
\author{Fabienne Chouraqui}

\date{}

\maketitle
\begin{abstract}
Let $G$ be a group and $H_1$,...,$H_s$ be subgroups of $G$ of  indices $d_1$,...,$d_s$ respectively. In 1974, M. Herzog and J. Sch\"onheim conjectured that if $\{H_i\alpha_i\}_{i=1}^{i=s}$,  $\alpha_i\in G$, is a coset partition of $G$, then $d_1$,..,$d_s$ cannot be distinct. We consider the  Herzog-Sch\"onheim conjecture for free groups of finite rank. We define $Y$ the space of coset partitions of $F_n$  and show $Y$ is a metric space with interesting properties.  In a previous paper, we gave some sufficient conditions on the coset partition of $F_n$ that ensure the  conjecture is satisfied. Here, we show that each coset partition of $F_n$, which satisfies one of these conditions,  has a neighborhood $U$ in $Y$ such that all the partitions in $U$ satisfy also the conjecture.
\end{abstract}
\maketitle
\section{Introduction}
Let $G$ be a group and $H_1$,...,$H_s$ be subgroups of $G$.  If there exist  $\alpha_i\in G$ such that $G= \bigcup\limits_{i=1}^{i=s}H_i\alpha_i$, and the sets  $H_i\alpha_i$, $1 \leq i \leq s$,  are pairwise disjoint, then  $\{H_i\alpha_i\}_{i=1}^{i=s}$ is \emph{a coset partition of $G$}  (or a \emph{disjoint cover of $G$}). We denote by $d_1$,...,$d_s$ the indices of $H_1$,...,$H_s$ respectively. The coset partition $\{H_i\alpha_i\}_{i=1}^{i=s}$ has  \emph{multiplicity} if $d_i=d_j$ for some $i \neq j$. The  Herzog-Sch\"onheim conjecture is true for the group $G$, if any coset partition of $G$ has multiplicity.

\setlength\parindent{10pt}In 1974, M. Herzog and J. Sch\"onheim conjectured that if $\{H_i\alpha_i\}_{i=1}^{i=s}$,  $\alpha_i\in G$, is a  coset partition of $G$, then $d_1$,..,$d_s$ cannot be distinct. In the 1980's, in a series of papers,  M.A. Berger, A. Felzenbaum and A.S. Fraenkel studied  the Herzog-Sch\"onheim conjecture \cite{berger1, berger2,berger3} and in \cite{berger4} they proved the conjecture is true for the pyramidal groups, a subclass of the finite solvable groups.  Coset partitions of finite groups with additional assumptions on the subgroups of the partition have been extensively studied. We refer to \cite{brodie,tomkinson1, tomkinson2,sun} and also to \cite{sun-site}. In \cite{schnabel}, the authors very recently proved that the conjecture is true for all groups of order less than $1440$. 

	 In \cite{chou}, we   study the   Herzog-Sch\"onheim conjecture in  free groups of finite rank, and develop a new approach based on the machinery of covering spaces. The  fundamental group of the  bouquet with $n$ leaves (or the wedge sum of $n$ circles),  $X$,  is $F_n$, the  free group of finite rank $n$. For any  subgroup $H$ of $F_n$ of finite index $d$, there exists  a $d$-sheeted covering space  $(\tilde{X}_H,p)$  with a fixed basepoint, which is also a combinatorial object.  Indeed, the underlying graph of $\tilde{X}_H$  is a directed labelled graph, with $d$ vertices,  that can be seen as  a finite complete bi-deterministic automaton; fixing the start and the end state at the basepoint, it recognises the set of elements in $H$.  It is   called \emph{the Schreier coset diagram  for $F_n$ relative to the subgroup  $H$} \cite[p.107]{stilwell} or  \emph{the Schreier automaton for $F_n$ relative to the subgroup $H$} \cite[p.102]{sims}. \\

 In $\tilde{X}_H$, the $d$ vertices (or states) correspond to the $d$ right cosets of $H$,  each edge (or transition) $Hg \xrightarrow{a}Hga$, $g \in F_n$, $a$ a generator of $F_n$,  describes the right action of $a$ on  $Hg$. We call  $\tilde{X}_H$,  \emph{the  Schreier graph  of $H$},  where  the $d$  vertices  $\tilde{x}_0, \tilde{x}_1,...,\tilde{x}_{d-1}$ are identified with the corresponding $d$ cosets  of $H$.  \emph{The transition group $T$} of the Schreier automaton for $F_n$ relative to  $H$  describes the action of $F_n$ on the set of the $d$ right cosets of $H$, and is generated by $n$ permutations. The group    $T$ is a subgroup of $S_d$ such that   $T \simeq\,^{F_n}\big/_{N_H}$, where $N_H= \bigcap\limits^{}_{g \in F_n}g^{-1}Hg$ is the normal core of $H$.  \\ Let  $\{H_i\alpha_i\}_{i=1}^{i=s}$ be a coset  partition of $F_n$, $n \geq 2$,  with $H_i<F_n$ of index $d_i>1$, $\alpha_i \in F_n$, $1 \leq i \leq s$. Let $\tilde{X}_{i}$ be the   Schreier  graph  of $H_i$, with transition group $T_i$, $1 \leq i \leq s$.  In \cite{chou}, we give some sufficient conditions on the transition groups of the Schreier  graphs $\tilde{X}_{i}$, $1 \leq i \leq s$, that ensure the coset partition $\{H_i\alpha_i\}_{i=1}^{i=s}$ has multiplicity. We state the following Theorems from \cite{chou} that we need for the paper:
 \begin{thmx}\label{theo0}\cite[Theorem 1]{chou}
 Let $F_n$ be the free group on $n \geq 2$ generators. Let $\{H_i\alpha_i\}_{i=1}^{i=s}$ be a coset  partition of $F_n$ with $H_i<F_n$ of index $d_i$, $\alpha_i \in F_n$, $1 \leq i \leq s$, and $1<d_1 \leq ...\leq d_s$.  Let $\tilde{X}_{i}$ denote the  Schreier  graph of $H_i$, with  transition group  $T_i$, $1 \leq i\leq s$.
     If there exists a $d_s$-cycle in $T_s$, then  the index $d_s$ appears in the partition at least  $p$ times, where $p$ is the smallest prime dividing $d_s$.  
  \end{thmx} 
The transition group    $T_s$ is a subgroup of the symmetric group $S_{d_s}$, generated by $n \geq 2$ permutations. Dixon proved that the probability that a random pair of  elements of $S_n$ generate $S_n$ approaches $3/4$ as $n \rightarrow\infty$, and the probability that they generate $A_n$ approaches $1/4$ \cite{dixon}. As  $d_s\rightarrow\infty$, the probability that $T_s$ is the symmetric group $S_{d_s}$ approaches $3/4$. So, asymptotically,  the probability that there exists a $d_s$-cycle in $T_s$ is greater than $3/4$. If  $T_s$ is cyclic,  there exists a $d_s$-cycle in $T_s$, since $d_s$ divides the order of $T_s$. That is, Theorem \ref{theo0} is satisfied with very high probability and the conjecture is ``asymptotically satisfied with probability greater than $3/4$''  for free groups of finite rank. \\

 Theorem \ref{theo1} provides a list of conditions on a coset partition  that ensure multiplicity.  Let $w \in F_n$. We denote by $o_{*i}(w)$ the minimal natural number, $1 \leq o_{*i}(w) \leq d_i$, such that $w^{o_{*i}(w)}$  is  a loop  at the vertex $H_i\alpha_i$ in $\tilde{X}_{i}$. Let  $o_{max}(w)=max\{ o_{*i}(w) 1 \leq i \leq s\}$ and $k=max\{o_{max}(v) \mid v \in F_n\}$,  $k$ is the  maximal length of a cycle in $\bigcup\limits_{i=1}^{i=s}T_i$. Let $p$ denote the smallest prime dividing $k$.  We show there exists  $u \in F_n$ such that  $o_{max}(u)=k$ and $\# \geq 2$, where   $\#=\mid \{1 \leq i\leq s \mid o_{*i}(u)=k\}\mid$. Using this notation, we show the following result, which implies under the assumption  $k > d_1$,  that there is a finite number of  cases  not covered by Theorem \ref{theo1}. 
   \begin{thmx}\label{theo1}\cite[Theorem 2]{chou}
  Let $F_n$ be the free group on $n \geq 2$ generators. Let $\{H_i\alpha_i\}_{i=1}^{i=s}$ be a coset  partition of $F_n$ with $H_i<F_n$ of index $d_i$, $\alpha_i \in F_n$, $1 \leq i \leq s$, and $1<d_1 \leq ...\leq d_s$.  Let $r$ be an integer, $2 \leq r \leq s-1$.
     If $k$, $p$ and $\#$, as defined above, satisfy one of the following conditions:
  \begin{enumerate}[(i)]
  \item  $k>d_{s-2}$.
 \item    $k>d_{s-3}$,  $p\geq 3$.
   \item    $k>d_{s-3}$,  $p=2$, and $\#=2$ or $\# \geq 4$.
   \item  $k>d_{s-r}$ and   $p\geq r$, or  $\#=p$, or   $\# \geq r+1$.  
    \end{enumerate} 
  Then the coset partition $\{H_i\alpha_i\}_{i=1}^{i=s}$ has multiplicity. 
  \end{thmx}

 Inspired by \cite{sikora}, in which the author defines the space of left orders of a left-orderable group and show it is a compact and totally disconnected metric space,  we define $Y$ to be  the space of coset partitions of $F_n$ (under some equivalence relation) and  show $Y$ is a metric space. In our case, the metric defined  induces the discrete topology. 
 \begin{thm_A}\label{theo2}
   Let $F_n$ be the free group on $n \geq 2$ generators. Let $Y$  be  the space of coset partitions of $F_n$ (under some equivalence relation). 
   Then $Y$ is a metric space with a metric $\rho$ and $Y$ is   (topologically) discrete 
   \end{thm_A}
We extend the results from \cite{chou} and  show that for  each coset partition of $F_n$, which satisfies one of the conditions in Theorems \ref{theo0} or \ref{theo1},  there exists a  neighborhood $U$ in $Y$ such that all the coset partitions in $U$ have multiplicity.
  \begin{thm_A}\label{theo3}
   Let $F_n$ be the free group on $n \geq 2$ generators. Let $Y$  be  the space of coset partitions of $F_n$ (under some equivalence relation) with metric $\rho$. Let $P_0=\{H_i\alpha_i\}_{i=1}^{i=s}$ be in $Y$, with  $1<d_1 \leq ...\leq d_s$. 
\begin{enumerate}[(i)]
\item If  $P_0$ satisfies the condition of Theorem \ref{theo0}, then every $P \in Y$ with $\rho(P, P_0)< \frac{1}{2}$ satisfies  the same condition and hence has multiplicity.
\item If  $P_0$ satisfies $(i)$ or $(ii)$ of Theorem \ref{theo1}, with some $2 \leq r \leq s-1$, then every $P \in Y$ with $\rho(P, P_0)< 2^{-(r+1)}$ satisfies  the same condition and hence has multiplicity.
\item If  $P_0$ satisfies $(iii)$ or $(iv)$ of Theorem \ref{theo1}, with some $2 \leq r \leq s-1$, then every $P \in Y$ with $\rho(P, P_0)< 2^{-(r+1)}$ has multiplicity.
\end{enumerate} 
 \end{thm_A}
 
    The paper is organized as follows.  In the first section, we introduce the space of coset partitions of $F_n$, an  action of $F_n$ on it, a metric  and prove Theorem \ref{theo2}.  We also give another proof of \cite[Theorem 3]{chou} using the action defined.   In Section $2$, we prove Theorem \ref{theo3}.     

\section{The space of coset partitions of $F_n$}\label{sec_space_partitions}
\subsection{Action of $F_n$ on the space of its coset partitions}\label{subsec_action_F_n-on-space_partitions}

 Let $F_n$ be the free group on $n \geq 2$ generators.  We define $Y'$ to be  the space of coset partitions of $F_n$ (only with subgroups of finite index). For each subgroup $H$ of $F_n$ of  finite index $d>1$, there exists a partition of $F_n$ by the $d$ cosets of $H$.  Generally, if  $P \in Y'$, then $P=\{H_i\alpha_i\}_{i=1}^{i=s}$,  a coset  partition of $F_n$ with $H_i<F_n$ of index $d_i$, $\alpha_i \in F_n$, $1 \leq i \leq s$, and $1<d_1 \leq ...\leq d_s$.  To get  some intuition on  $Y'$, it is worth recalling that  the subgroup growth of $F_n$ is exponential. There exists a  natural right action of $F_n$ on $Y'$. Indeed, if $w \in F_n$, then $P \cdot w = P'$, with $P'= \{H_i\alpha_i\,w\}_{i=1}^{i=s}$ in $Y'$. 
\begin{lem}
The natural right  action of $F_n$ on $Y'$ is faithful 
\end{lem}
\begin{proof}
 Let $w \in F_n$. Then  $P \cdot w = P$ for every $P \in Y'$  if and only if  $w$ belongs to the intersection of all the subgroups of finite index of $F_n$. As $F_n$ is residually finite \cite[p.158]{robinson},  the intersection of all the subgroups of finite index of $F_n$ is trivial, so $w=1$, that is the action is faithful.
\end{proof}
Let $P=\{H_i\alpha_i\}_{i=1}^{i=s}$ in $Y'$ and let $\tilde{X}_i$ be the Schreier graph of $H_i$, $ 1 \leq i \leq s$ (as defined in the introduction). Let $w \in F_n$. We denote by $o_{*i}(w)$ the minimal natural number, $1 \leq o_{*i}(w) \leq d_i$, such that $w^{o_{*i}(w)}$  is  a loop  at the vertex $H_i\alpha_i$ in $\tilde{X}_{i}$ or equivalently $w^{o_{*i}(w)} \in \alpha_i^{-1}H_i\alpha_i$ \cite[Section 4.1]{chou}.
\begin{lem}
Let  $P=\{H_i\alpha_i\}_{i=1}^{i=s}$ in $Y'$. Then $\mid Orb_{F_n}(P) \mid \leq d_1...d_s$, where $ Orb_{F_n}(P)$ denotes the orbit of $P$ under the action of $F_n$.
Furthermore, for $w \in F_n$,  $\mid Orb_{w}(P) \mid =lcm(o_{*1}(w),...,o_{*s}(w))$, where $ Orb_{w}(P)$ denotes the orbit of $P$ under the action of $\langle w \rangle$.
\end{lem}
\begin{proof}
From the definition of the action of $F_n$ on $P$, $F_n$ permutes between the cosets of $H_1$, between the cosets of $H_2$ and so on. So, $\mid Orb_{F_n}(P) \mid \leq d_1...d_s$. The size of $ Orb_{w}(P)$ is equal to the minimal natural number such that  $P \cdot w^k = P$, that is  $k=lcm(o_{*1}(w),...,o_{*s}(w))$. Indeed,     $P \cdot w^k = P$,   if and only if  $w^k \in \bigcap\limits^{i=s}_{i=1}\alpha_i^{-1}H_i\alpha_i$, that is if and only if $lcm(o_{*1}(w),...,o_{*s}(w))$ divides $k$ \cite[Lemma 4.10]{chou}. 
As $k=lcm(o_{*1}(w),...,o_{*s}(w))$ is minimal such that  $P \cdot w^k = P$,   $\mid Orb_{w}(P) \mid =lcm(o_{*1}(w),...,o_{*s}(w))$.
\end{proof}
In \cite[Theorem 3]{chou},  we give a  condition on a partition $P$ that ensures  the same  subgroup appears at least twice in $P$.
We state a shortened version of the result and give another proof using the action of $F_n$ on $Y'$.
   \begin{thmx}\cite[Theorem 3]{chou}\label{theo4}
Let  $P=\{H_i\alpha_i\}_{i=1}^{i=s}$ in $Y'$. If there exist  $1 \leq j,k \leq s$ such that $\bigcap \limits_{i=1}^{i=s}\alpha_i^{-1}H_i\alpha_i \subsetneqq \bigcap \limits_{i\neq j,k}\alpha_i^{-1}H_i\alpha_i$.  Then  $H_j=H_k$.
  \end{thmx}
\begin{proof}
From the assumption, there exists  $w \in  \bigcap \limits_{i\neq j,k}\alpha_i^{-1}H_i\alpha_i$, $w \notin \bigcap \limits_{i=1}^{i=s}\alpha_i^{-1}H_i\alpha_i$. Then $P \cdot w=\;\{H_i\alpha_iw\}_{i=1}^{i=s}$ gives $F_n=\bigcup\limits_{i \neq j,k} H_i\alpha_i \cup H_j\alpha_jw\cup H_k\alpha_kw$.
So, $H_j\alpha_jw\cup H_k\alpha_kw=H_j\alpha_j\cup H_k\alpha_k$, with $H_j\alpha_jw \neq H_j\alpha_j$ and $H_k\alpha_kw\neq H_k\alpha_k$. As $H_j\alpha_jw \cap H_j\alpha_j =\emptyset$  and $H_k\alpha_kw\cap H_k\alpha_k =\emptyset$, $H_j\alpha_jw \subseteq H_k\alpha_k $ and $H_k\alpha_kw \subseteq H_j\alpha_j$. From $H_j\alpha_jw\cup H_k\alpha_kw=H_j\alpha_j\cup H_k\alpha_k$ again, we have $H_j\alpha_jw= H_k\alpha_k$ and $H_j\alpha_j= H_k\alpha_kw$, that is $H_k\alpha_kw\alpha_j^{-1}=H_j$ a subgroup, so $H_k\alpha_kw\alpha_j^{-1}=H_k$, that is $H_k=H_j$. Further, we recover  $o_{*j}(w)=o_{*k}(w)=2$ and $w^2 \in  \bigcap \limits_{i=1}^{i=s}\alpha_i^{-1}H_i\alpha_i$, as in the proof of \cite[Theorem 3]{chou}.
\end{proof} 
We recall that for each subgroup $H$ of index $d$ in $F_n$ (or in any group), there is a transitive action of the group on the set of right cosets of $H$, that is given two cosets $H\alpha$ and  $H\beta$ of $H$, there exists $w$ such that $H\alpha\cdot w=\,H\beta$. So, the following question arises:
\begin{ques}\label{ques}
Let $P=\{H_i\alpha_i\}_{i=1}^{i=s}$ and  $P'=\{H_i\beta_i\}_{i=1}^{i=s}$ in $Y'$. Does  there necessarily exist $w \in F_n$ such that $P'=P\cdot w$ ?
\end{ques}

\subsection{Topology in the space of coset partitions of $F_n$}
We refer to \cite{munkres} for more details.
Let  $Y'$  be  the space of coset partitions of $F_n$. Given $P=\{H_i\alpha_i\}_{i=1}^{i=s}$ in $Y'$, with $d_s \geq ...\geq d_1>1$, we identify $P$ with the $s$-tuple $(H_s,...,H_1)$ and we consider $H_s$ at the first place, $H_{s-1}$ at the second place and so on. Let  $P'\in Y'$, $P'= \{K_i\beta_i\}_{i=1}^{i=t}$. 
We define a function $d: Y' \times Y' \rightarrow\mathbb{R}$ :
\[ d(P,P')= \left\{\begin{array}{ccc}
2^{-k} & & \text{if } k\, \text{is  the first place at which } K_i \neq H_i\\
0 & & \text{if }  t=s;\; \;\;H_i=K_i, \,\;\forall 1 \leq i \leq s \\
\end{array} \right.
\]
The  function $d: Y' \times Y' \rightarrow\mathbb{R} \cup \{\infty\}$ is a \emph{semi-metric} if for all $P,P',P'' \in Y'$, $d$  satisfies  $d(P,P')=d(P',P)$ (symmetry) and  $d(P,P'')\, \leq \,d(P,P')\,+d(P',P'')$ 
(triangle inequality). A standard argument shows:

\begin{lem}
The function $d$ is a semi-metric.
\end{lem}
\begin{proof}
Let $P,P',P'' \in Y'$. Clearly,  $d(P,P')=d(P',P)$. Assume  $d(P,P')=2^{-k}$,  $d(P',P'')=2^{-\ell}$, and $d(P,P'')=2^{-m}$. If $k>1$ or $\ell>1$, then   $m=min\{k, \ell\}$  and   $d(P,P'')=2^{-(min\{k,\ell\})} \, \leq \,2^{-k}\,+2^{-\ell}$. If $k=\ell=1$, then $m \geq 1$ and $d(P,P'')=2^{-m}<1$.
\end{proof}
A \emph{metric} is a semi-metric with 
the additional requirement that $d(P,P') = 0$  implies $P=P'$.
Identifying points with zero distance in a semi-metric $d$ is an equivalence relation that leads to a  metric $\hat{d}$.    The function $\hat{d}$ is then a metric in  $Y'\big/\equiv$, with $P\equiv P'$ if and only if $d(P,P') = 0$. If the answer to Question \ref{ques} is positive, then 
 $Y'\big/\equiv$ is the same as the quotient of $Y'$ by the action of $F_n$. We denote $Y'\big/\equiv$  by $Y$ and  $\hat{d}$ by  $\rho$.

We denote by $B_r(P_0)=\{P \in Y \mid \rho(P,P_0)<2^{-r} \}$, the open ball of radius $2^{-r}$ centered at $P_0$. A set $U \subset Y$  is open if and only if for every point $P \in U$,  there exists $\epsilon>0$ such that $B_\epsilon(P)\subset U$. A space $Y$ is \emph{totally disconnected} if every two distinct points of $Y$ are contained in two disjoint open sets covering the space.
A point $P$ in a metric space   $Y$ is 
 an \emph{isolated point} of $Y$ if there exists a real number $\epsilon>0$, such that $B_\epsilon(P)=\{P\}$. If all the points in  $Y$ are isolated, then $Y$ is \emph{discrete}. The space $Y$  is  \emph{(topologically) discrete} if $Y$ is  discrete as a topological space, that is the  metric  may be different from the  discrete metric.
\begin{thm}
The metric space $Y$ is (topologically) discrete.
\end{thm}
\begin{proof}
We show that all the points in $Y$ are isolated. Let 
$P=\{H_i\alpha_i\}_{i=1}^{i=s}$ in $Y$, with $d_s \geq ...\geq d_1>1$. Then for  $\epsilon < 2^{-(s+1)}$, $B_\epsilon(P)=\{P\}$.
\end{proof}
This implies that  $Y$ is Hausdorff, bounded and totally disconnected,
 facts that could be easily proved directly using $\rho$. A metric space $X$ is \emph{uniformly discrete}, if there exists $\epsilon > 0$ such that for any $x,x' \in X$, $x \neq x'$, $\rho(x,x')>\epsilon$. The space $Y$ is not uniformly discrete.  
 \begin{rem}
 Given an arbitrary group $G$, one can define in the same way the space $Y$, the metric $\rho$ and obtain the same topological properties. The action of $G$ on $Y$ can also be defined in the same way, but it is not necessarily faithful anymore.
 \end{rem}

\section{Extension of Theorem \ref{theo1}: Proof of Theorem \ref{theo3}}\label{sec_main_result}
 
\begin{proof}[Proof of Theorem \ref{theo3}]

Let $P_0=\{H_i\alpha_i\}_{i=1}^{i=s}$,  with  $1<d_{H_1} \leq ...\leq d_{H_s}$. Let $P \in Y$, $P=\{K_i\beta_i\}_{i=1}^{i=t}$, with  $1<d_{K_1} \leq ...\leq d_{K_t}$.\\
$(i)$  If $\rho(P, P_0)< \frac{1}{2}$, then $K_t=H_s$. So, if  there exists a $d_s$-cycle in $T_{H_s}$, the index $d_s$ appears in $P_0$ and in $P$ at least $p$ times, where $p$ is the least prime dividing $d_s$. Note that   this implies necessarily  $\rho(P, P_0) \leq  2^{-p-1}$.\\
$(ii)$, $(iii)$ If $\rho(P, P_0)< 2^{-(r+1)}$, $2 \leq r \leq s-1$, then $K_t=H_s$,  $K_{t-1}=H_{s-1}$,...,  $K_{t-r}=H_{s-r}$. 
If $k$,  the maximal length of a cycle in $\bigcup\limits_{i=1}^{i=s}T_{H_i}$,  satisfies   $k>d_{H_{s-r}}$, then  $k>d_{K_{t-r}}$ also. Furthermore, $k$ occurs in  $\bigcup\limits_{i=s-r}^{i=s}T_{H_i}$ and also in $\bigcup\limits_{j=t-r}^{j=t}T_{K_j}$, since  $\bigcup\limits_{i=s-r}^{i=s}T_{H_i}\,=\,
\bigcup\limits_{j=t-r}^{j=t}T_{K_j}$. 
If $P_0$ satisfies condition $(i)$ or $(ii)$ of Theorem \ref{theo1}, then $P$   satisfies  the same condition and hence has multiplicity. 

 If $P_0$ satisfies condition $(iii)$ or $(iv)$ of Theorem \ref{theo1}, then $P_0$  has multiplicity, with $d_{H_{s-i}}=d_{H_{s-j}}$ for some $i \neq j$, $0 \leq i,j\leq r$. So, $d_{K_{t-i}}=d_{K_{t-j}}$ also, that is $P$ has multiplicity.
\end{proof}

To conclude, we ask the  following natural question: does there exist a metric on the space of coset partitions $Y'$ that induces a non-discrete topology and yet can give rise to a result of the form of Theorem \ref{theo3} ?

\bigskip\bigskip\noindent
{ Fabienne Chouraqui,}

\smallskip\noindent
University of Haifa at Oranim, Israel.
\smallskip\noindent
E-mail: {\tt fabienne.chouraqui@gmail.com} {\tt fchoura@sci.haifa.ac.il}


\begin{thebibliography}{40}

\bibitem{berger1}M.A. Berger, A. Felzenbaum, A.S. Fraenkel, \emph{Improvements to two results concerning systems of residue sets}, Ars. Combin. {\bf 20} (1985), 69-82.
\bibitem{berger2}M.A. Berger, A. Felzenbaum, A.S. Fraenkel, \emph{The Herzog-Sch\"onheim conjecture for finite nilpotent groups}, Canad. Math. Bull. {\bf 29}(1986),329-333.
\bibitem{berger3}M.A. Berger, A. Felzenbaum, A.S. Fraenkel, \emph{Lattice parallelotopes and disjoint covering systems}, Discrete Math. {\bf 65} (1987), 23-44.
\bibitem{berger4}M.A. Berger, A. Felzenbaum, A.S. Fraenkel, \emph{Remark on the multiplicity of a partition of a group into cosets}, Fund. Math. {\bf 128} (1987), 139-144.
\bibitem{brodie} M.A. Brodie, R.F. Chamberlain, L.C Kappe, \emph{Finite coverings by normal subgroups}, Proc. Amer. Math. Soc. {\bf 104} (1988), 669-674.
\bibitem{chou} F. Chouraqui, \emph{The Herzog-Sch\"onheim conjecture for finitely generated groups}, ArXiv 1803.08301.
\bibitem{dixon}J. D. Dixon, The probability of generating the symmetric group, Math. Z. {\bf 110} (1969), 199-205.
\bibitem{herzog} M. Herzog, J. Sch\"onheim, \emph{Research problem no. 9}, Canad. Math. Bull., { \bf 17 } (1974), 150.
\bibitem{schnabel}L. Margolis, O. Schnabel, \emph{The Herzog-Sch\"onheim conjecture for small   groups and harmonic subgroups} , ArXiv 1803.03569.
\bibitem{munkres}J. Munkres, \emph{Topology}, Pearson Modern Classics for Advanced Mathematics Series, 2000.
 \bibitem{por4}\u{S}. Porubsk\'y, J. Sch\"onheim, \emph{Covering systems of Paul Erd\H{o}s. Past, present and future. Paul Erd\H{o}s and his mathematics}, J$\acute{a}$nos Bolyai Math. Soc., { \bf 11 }  (2002), 581-627.
\bibitem{robinson}D.J.S. Robinson, \emph{A Course in the Theory of Groups}, Graduate Texts in Mathematics {\bf 80}, Springer-Verlag,  Berlin, Heidelberg, New York (1980). 
\bibitem{rotman}J.J. Rotman, \emph{An Introduction to Algebraic Topology}, Graduate Texts in Mathematics {\bf 119}, Springer-Verlag,  Berlin, Heidelberg, New York (1988).
\bibitem{sikora}A. Sikora, \emph{Topology on the spaces of orderings of groups}, Bull. London Math. Soc. {\bf 36} (2004), 519-526.
\bibitem{sims} C.C. Sims, \emph{Computation with finitely presented groups}, Encyclopedia of Mathematics and its Applications { \bf 48 },  Cambridge University Press (1994).
\bibitem{stilwell}J. Stillwell, \emph{Classical Topology and Combinatorial Group Theory}, Graduate Texts in Mathematics {\bf 72}, Springer-Verlag,  Berlin, Heidelberg, New York (1980).

\bibitem{sun} Z.W. Sun, \emph{Finite covers of groups by cosets or subgroups}, Internat. J. Math. { \bf 17 } (2006), n.9, 1047-1064.
\bibitem{sun-site} Z.W. Sun, \emph{Classified publications on covering systems}, $http://math.nju.edu.cn/~zwsun/Cref.pdf$.

\bibitem{tomkinson1} M.J. Tomkinson, \emph{Groups covered by abelian subgroups}, London Math. Soc. Lecture Note Ser. {\bf 121},  Cambridge Univ. Press (1986).
\bibitem{tomkinson2} M.J. Tomkinson, \emph{Groups covered by finitely many cosets or subgroups}, Comm. Algebra {\bf 15}(1987), 845-859. 
  \end{thebibliography}
\end{document}